\documentclass{article}
\usepackage[T1]{fontenc}
\usepackage[utf8]{inputenc}
\usepackage{microtype}
\usepackage{fullpage}
\usepackage[parfill]{parskip}
\usepackage{mathtools,amssymb,amsthm}
\usepackage[dvipsnames]{xcolor}
\usepackage{comment}
\usepackage{enumitem}
\setlist[enumerate,1]{label={\textup{(\roman*)}}}
\usepackage[hypertexnames=false]{hyperref}
\hypersetup{
    colorlinks,
    linkcolor={red!60!black},
    citecolor={green!60!black},
    urlcolor={blue!60!black}
}
\DeclareMathOperator{\rank}{rank}
\DeclareMathOperator{\wt}{wt}

\theoremstyle{plain}
\newtheorem{theorem}{Theorem}[section]
\newtheorem{proposition}[theorem]{Proposition}
\newtheorem{lemma}[theorem]{Lemma}

\newtheorem{conjecture}[theorem]{Conjecture}
\theoremstyle{definition}

\newtheorem{observation}[theorem]{Observation}

\numberwithin{equation}{section}

\title{Oddtown and Eventown Theorems for Lattice Paths}

\author{
    Umesh Shankar\\
    Department of Computer Science and Automation,\\
    Indian Institute of Science Bengaluru,\\
    Bengaluru 560012, Karnataka, India\\
    Email: \texttt{204093001@iitb.ac.in}, \texttt{umeshshankar@outlook.com}
}

\date{June 13, 2025}

\begin{document}
\maketitle

\begin{abstract}
For North-East lattice paths (which we simply call lattice paths), we define intersection in terms of common edges. We prove that a family of paths from $(0,0)$ to $(n,n)$ in which every two distinct paths have an even number of common edges has size at most $2^n$, and that this bound is attained. If $M_{\mathrm{odd}}(n)$ denotes the maximum size of a family in which every two distinct paths have an odd number of common edges, then we prove
\[
    M_{\mathrm{odd}}(n)\le n(n-1)+1
\]
and construct families showing that $M_{\mathrm{odd}}(n)=\Theta(n^2)$. Finally, we construct at least $C_n$ distinct extremal even-intersecting families, where $C_n$ is the $n$th Catalan number, and conjecture that these are all the extremal families.
\end{abstract}

\section{Introduction}

Let $n$ be a positive integer and let $[n]=\{1,2,\dots,n\}$. The classical Oddtown theorem \cite{BabaiFrankl1992} states the following.

\begin{theorem}[Oddtown]
Let $\mathcal F\subseteq 2^{[n]}$ be a family such that $|F|$ is odd for every $F\in\mathcal F$ and $|F\cap G|$ is even for every two distinct $F,G\in\mathcal F$. Then $|\mathcal F|\le n$.
\end{theorem}

The classical Eventown theorem \cite{berle} is its even analogue.

\begin{theorem}[Eventown]
Let $\mathcal F\subseteq 2^{[n]}$ be a family such that $|F|$ is even for every $F\in\mathcal F$ and $|F\cap G|$ is even for every $F,G\in\mathcal F$. Then
\[
    |\mathcal F|\le 2^{\lfloor n/2\rfloor}.
\]
\end{theorem}

A \emph{North-East lattice path}, or simply a \emph{lattice path}, is a finite sequence of vertices
\[
    F=(v_0,v_1,\dots,v_k)\in(\mathbb Z^2)^{k+1}
\]
such that
\[
    v_i-v_{i-1}\in\{(0,1),(1,0)\}
    \qquad (1\le i\le k).
\]
Equivalently, a lattice path is a word in the alphabet $\{N,E\}$, where $N=(0,1)$ and $E=(1,0)$. Let $L(m)$ denote the set of all $m$-step lattice paths starting at $(0,0)$.

For $F=(v_0,\dots,v_k)$, define its edge set by
\[
    E(F)=\bigl\{\{v_i,v_{i+1}\}:0\le i<k\bigr\}.
\]
For two paths $F$ and $G$, both starting at $(0,0)$, define
\[
    F\cap G:=E(F)\cap E(G).
\]
Thus $|F\cap G|$ is the number of common edges of the two paths. \footnote{This notion and the study of intersection theorems was suggested by Hiranya Dey through personal communication.}

The following elementary proposition illustrates this notion of intersection.

\begin{proposition}\label{prop:nonempty}
Let $n\ge1$, and let $\mathcal F$ be a family of lattice paths from $(0,0)$ to $(n,n)$ such that $F\cap G\ne\varnothing$ for every two distinct $F,G\in\mathcal F$. Then
\[
    |\mathcal F|\le \binom{2n-1}{n-1},
\]
and the bound is attained.
\end{proposition}

\begin{proof}
For each lattice path $F$ that begins with a $N$ step, consider the pair $\{F, F^c\}$ where $F^c$ is the path $F$ reflected along the line $x=y$. These pairs form a partition of the set of all lattice paths from $(0,0)$ to $(n,n)$. Since $F\cap F^c=\emptyset$, we can pick up at most one element from each pair

Consequently, $\mathcal F$ contains at most one path from each reflection pair, and hence
\[
    |\mathcal F|\le \frac12\binom{2n}{n}=\binom{2n-1}{n-1}.
\]
The family of all paths whose first step is $N$ has this size, and any two of its members share their first edge.
\end{proof}

We prove lattice-path analogues of Oddtown and Eventown. Our first main result is the following.

\begin{theorem}\label{thm:eventown}
Let $n$ be a positive integer, and let $\mathcal F$ be a family of lattice paths from $(0,0)$ to $(n,n)$ such that $|F\cap G|$ is even for every two distinct $F,G\in\mathcal F$. Then
\[
    |\mathcal F|\le 2^n,
\]
and the bound is attained.
\end{theorem}

For the odd-intersection problem, define
\[
M_{\mathrm{odd}}(n):=
\max\bigl\{|\mathcal F|:\mathcal F\text{ is a family of paths from $(0,0)$ to $(n,n)$, and }|F\cap G|\text{ is odd for all }F\ne G\bigr\}.
\]

\begin{theorem}\label{thm:oddtown}
For every positive integer $n$,
\[
    M_{\mathrm{odd}}(n)\le n(n-1)+1.
\]
Moreover,
\[
    M_{\mathrm{odd}}(2k)\ge k^2
\]
for $k\ge1$, and
\[
    M_{\mathrm{odd}}(2k+1)\ge k(k+2)
\]
for $k\ge1$. In particular,
\[
    M_{\mathrm{odd}}(n)=\Theta(n^2).
\]
\end{theorem}

We conclude with a Catalan family of extremal examples for Theorem~\ref{thm:eventown}.

\section{Proof of Theorem~\ref{thm:eventown}}\label{sec:even}

We identify a path in $L(m)$ with a binary word of length $m$ by replacing $N$ by $1$ and $E$ by $0$. For $x\in L(m)$, let $\wt(x)$ be its Hamming weight. For lattice paths $x,y\in L(m)$, define
\[
B_m(x,y):=|x\cap y|\pmod2\in\mathbb F_2.
\]

In terms of the binary words associated with lattice paths, $|x\cap y|\pmod2$ can be interpreted as follows:

\begin{equation}
    |x\cap y|\pmod2 = \# \{i:x_i=y_i \mbox{ and } \sum_{j=1}^{i-1}x_j=\sum_{j=1}^{i-1}y_j\} \pmod{2}
\end{equation}

Let $F=f_1\cdots f_{2n}$ be a path from $(0,0)$ to $(n,n)$. Define
\[
    u(F)=(f_1,\dots,f_n),
    \qquad
    v(F)=(1-f_{2n},\dots,1-f_{n+1}).
\]
Since $F$ contains exactly $n$ ones, we have
\[
    \wt(u(F))=\wt(v(F)).
\]

We state the following observation without proof.
\begin{observation}\label{obs:parity}
For any two paths $F,G$ from $(0,0)$ to $(n,n)$,
\begin{equation}\label{eq:split-parity}
B_{2n}(F,G)=B_n(u(F),u(G))+B_n(v(F),v(G))
\qquad\text{in }\mathbb F_2.
\end{equation}
\end{observation}


If $B_{2n}(F,G)$ is even, then $B_n(u(F),u(G))=B_n(v(F),v(G))$. Therefore, we have to prove the following.

\begin{lemma}\label{bip}
    Let $R\subseteq L(n)\times L(n)$ with $\wt(x)=\wt(y)$ for all $(x,y)\in R$ and $B_n(x,x')=B_n(y,y')$ for all $(x,y),(x',y')\in R$. Then $|R|\le 2^n$. 
\end{lemma}
\begin{proof}[Proof of Lemma \ref{bip}]
    Form a bipartite graph $G$ where the vertex sets are $L(n)$ on both parts. The edge set is given by the subset $R$. Let $A_j, B_j$ be the left and right vertex sets of the $j$th connected component of $G$. So, we have \begin{eqnarray}
        |R|\le \sum_j |A_j||B_j|\le (\sum_j |A_j|^2)^{1/2}(\sum_j |B_j|^2)^{1/2}
    \end{eqnarray}
    
    Let $x,x'$, on the same part, have a common neighbour $y$ in $G$. Then we have $B_n(x,x')=B_n(y,y)=n\pmod{2}$. So, for any two $x,x'$ from the same part with a path between them, we have $B_n(x,x')=n\pmod{2}$. Therefore, for the $j$th connected component of $G$, we have $B_n(x,x')=B_n(y,y')=n\pmod{2}$ for $x,x'\in A_j$ and $y,y'\in B_j$.

    Similarly, for two pairs $(x,y),(x',y)\in R\cap (A_{j_1}\times B_{j_1}), (a,b)\in R\cap (A_{j_2}\times B_{j_2})$ where $j_1\ne j_2$, we have 
    $$B_n(x,a)=B_n(y,b)=B_n(x',a).$$ By connectivity in the $j$th component, we have $$B_n(x,a)=B_n(x',a)$$ for any $x,x'\in A_{j_1}$ and $a\in A_{j_2}$.
    
    Therefore, to complete the proof, we need to prove the following lemma. 

\begin{lemma}\label{hcl}
    Let $A_1,\dots, A_r$ be pairwise disjoint subsets of $L(n)$ such that:
    \begin{enumerate}
        \item every $A_i$ belongs to one Hamming-weight level;
        \item $B_n(x,y)=n\pmod{2}$ for all $x,y\in A_i$;
        \item for $i\ne j$, $B_n$ is constant on $A_i\times A_j$.
    \end{enumerate}
    Then $\sum_j |A_j|^2\le 2^n$.
\end{lemma}
\begin{proof}[Proof of Lemma \ref{hcl}]
    We induct on $n$ and the case of $n=0$ is immediate.
    We can partition $A_i=B_i^0\cup B_i^1$ where $B_i^0$ are the words/lattice paths ending in $0$ and $B_i^1$ are the words/lattice paths ending in $1$.
    From $A_1,\dots,A_r$, we produce two families of subsets that satisfy the induction hypothesis for $n-1$. 
    
    Let $$\mathcal{G}_0=\{B_i^0\}_{i=1}^{r} \cup \{B_{i}^1: B_{i}^0\ne \emptyset\},\quad \mathcal{G}_1=\{B_i^1\}_{i=1}^{r} \cup \{B_{i}^0: B_{i}^1\ne \emptyset\}.$$
    We claim that $\mathcal{G}_0$ satisfies the induction hypothesis. The other one follows by symmetry.
    Let $u,v\in B_i^0$, then $B_n(u0,v0)=B_{n-1}(u,v)+B_1(1,1)=B_{n-1}(u,v)+1\equiv n\pmod2$. So, $B_{n-1}(u,v)\equiv n-1 \pmod2$.
    Now, if $u\in B_i^0,v\in B_j^0$, then $B_{n}(u0,v0)=B_{n-1}(u,v)$ is constant. Now, if $u\in B_i^0, v\in B_{j}^1$ ($i$ may be equal to $j$), then $B_{n}(u0,v1)=B_{n-1}(u,v)$ is constant. We only have to prove that $B_i^0\cap B_{i+1}^1 =\emptyset$ to show that these are all disjoint subsets. Suppose $u\in B_i^0 \cap B_{i+1}^1$, this implies that $u0\in A_i, u1\in A_{i+1}$. Also, there exists an $x\in B_{i+1}^0$ such that $x0\in A_{i+1}$. We know that $B_n(u0,x0)=B_{n-1}(u,x)=B_{n}(u1,x0)=n\pmod 2$ and $B_n(u0,u1)=n-1\pmod2$. But $B_na$ is constant on $A_{i}\times A_j$ and this forces a contradiction. This proves our claim.

    Let $I$ be the set indices such that both $B_i^0\ne \emptyset$ and $B_i^1\ne \emptyset$. Then, we can write, by the induction hypothesis,
    \begin{eqnarray}
        \sum_{i\notin I} |B_i^0|^2 + \sum_{i \in I} (|B_i^0|^2+|B_i^1|^2)&\le& 2^{n-1},\label{fam1}\\
        \sum _{i\notin I} |B_{i}^1|^2+\sum_{i \in I} (|B_i^0|^2+|B_i^1|^2)&\le& 2^{n-1}.\label{fam2}
    \end{eqnarray}
    We want to evaluate 
    \begin{equation}
        \sum_{i\notin I} (|B_i^0|+|B_i^1|)^2 + \sum_{i\in I} (|B_i^0|+|B_i^1|)^2.
    \end{equation}
    We have that if $i\notin I$, then $|B_i^0|+|B_i^1| = |B_i^0| \mbox{ or } |B_i^1|$ and for $i\in I$, we have $(|B_i^0|+|B_i^1|)^2 \le 2|B_i^0|^2 + 2|B_{i}^1|^2$.
    Using these two inequalities along with Equations \eqref{fam1} and \eqref{fam2}, we get 
    \begin{eqnarray*}
        \sum_{i\notin I} (|B_i^0|+|B_i^1|)^2 + \sum_{i\in I} (|B_i^0|+|B_i^1|)^2 &\le& \sum_{i\notin I} (|B_i^0|^2+|B_i^1|^2) + \sum_{i\in I} 2(|B_i^0|^2+|B_i^1|^2)\\ &\le& 2^{n-1}+2^{n-1}\\&\le& 2^n.
    \end{eqnarray*}
\end{proof}
This completes the proof.
\end{proof}
The tightness of the bound is discussed in Section \ref{sec:catalan}.


\section{Proof of Theorem~\ref{thm:oddtown}}\label{sec:odd}

\subsection{The upper bound}

Let $\Gamma_n$ be the $(n+1)\times(n+1)$ grid graph with vertex set $\{0,1,\dots,n\}^2$, and work in the edge space $\mathbb F_2^{E(\Gamma_n)}$ with the standard bilinear form
\[
    \langle x,y\rangle
    :=\sum_{e\in E(\Gamma_n)}x_e y_e.
\]
We identify each path with its edge-incidence vector. Thus
\[
    \langle F,G\rangle=|F\cap G|\pmod2.
\]

Let $W$ be the span of all lattice paths from $(0,0)$ to $(n,n)$. For $1\le i,j\le n$, let $C_{i,j}$ be the four-edge cycle bounding the cell whose lower-left vertex is $(i-1,j-1)$. Let $Q=N^nE^n$, the path that follows the left boundary and then the top boundary of the grid.

The vectors
\[
    Q,\quad C_{i,j}\quad(1\le i,j\le n)
\]
form a basis of $W$. Indeed, the cell cycles form a basis of the cycle space of $\Gamma_n$. The symmetric difference of any two $(0,0)$--$(n,n)$ paths is a cycle-space vector, so every path lies in the span of $Q$ and the cell cycles. Conversely, each $C_{i,j}$ is the symmetric difference of two North-East paths that agree everywhere except on the two routes around that cell; hence every cell cycle lies in $W$. Finally, $Q$ is not in the cycle space because its only odd-degree vertices are $(0,0)$ and $(n,n)$, so this spanning set is linearly independent.

Let $G_W$ be the Gram matrix of the bilinear form on $W$ in this basis. Then
\begin{equation}\label{eq:GW-block}
G_W=
\begin{pmatrix}
0&q^T\\
q&A
\end{pmatrix},
\end{equation}
where $A$ is the adjacency matrix over $\mathbb F_2$ of the $n\times n$ grid graph on the cells. Indeed, two distinct cell cycles have scalar product $1$ exactly when their cells share an edge, and every cell cycle has scalar product $0$ with itself.

The vector $q$ records the parity of the common edges of $Q$ and the cell cycles. Explicitly,
\[
q_{i,j}=1
\quad\Longleftrightarrow\quad
\bigl(i=1,\ j<n\bigr)
\text{ or }
\bigl(j=n,\ i>1\bigr).
\]
Let $u\in\mathbb F_2^{n^2}$ be defined by
\[
    u_{i,j}=\mathbf 1[i<j].
\]
A direct check of the four neighbours of each cell shows that
\begin{equation}\label{eq:Auq}
    Au=q.
\end{equation}
Replacing $Q$ by
\[
    Q':=Q+\sum_{i<j}C_{i,j}
\]
is a basis change. By \eqref{eq:Auq}, the vector $Q'$ is orthogonal to every cell cycle. The bilinear form is alternating on $W$, because each basis vector has even edge cardinality, so $\langle Q',Q'\rangle=0$. Therefore $G_W$ is congruent to
\begin{equation}\label{eq:GW-congruent}
    \begin{pmatrix}
    0&0\\
    0&A
    \end{pmatrix},
\end{equation}
and hence
\begin{equation}\label{eq:rankGW}
    \rank(G_W)=\rank(A).
\end{equation}

We now compute the rank of $A$.

\begin{lemma}\label{lem:grid-rank}
Over $\mathbb F_2$, the adjacency matrix of the $n\times n$ grid graph has rank $n(n-1)$.
\end{lemma}

\begin{proof}
Let $P$ be the adjacency matrix of the path graph on $n$ vertices. Write a vector in $\mathbb F_2^{n^2}$ as $n$ rows $r_1,\dots,r_n\in\mathbb F_2^n$, and set $r_0=r_{n+1}=0$. The equation $Ar=0$ is equivalent to
\begin{equation}\label{eq:row-recurrence}
    r_{s+1}=Pr_s+r_{s-1}
    \qquad(1\le s\le n).
\end{equation}
Thus a kernel vector is uniquely determined by its first row $r_1$.

Define polynomials over $\mathbb F_2$ by
\[
    f_0(x)=0,
    \qquad
    f_1(x)=1,
    \qquad
    f_{s+1}(x)=x f_s(x)+f_{s-1}(x).
\]
Recurrence \eqref{eq:row-recurrence} gives
\[
    r_s=f_s(P)r_1.
\]
The characteristic polynomial $\chi_n(x)$ of $P$ satisfies
\[
    \chi_0(x)=1,
    \qquad
    \chi_1(x)=x,
    \qquad
    \chi_s(x)=x\chi_{s-1}(x)+\chi_{s-2}(x),
\]
where the plus sign appears because we are in characteristic $2$. Hence
\[
    f_{n+1}(x)=\chi_n(x).
\]
By the Cayley--Hamilton theorem,
\[
    r_{n+1}=f_{n+1}(P)r_1=\chi_n(P)r_1=0
\]
for every choice of $r_1$. Therefore every $r_1\in\mathbb F_2^n$ determines a kernel vector, and it does so uniquely. Hence $\dim\ker A=n$, and
\[
    \rank(A)=n^2-n=n(n-1).
\]
\end{proof}

Now let $\mathcal F=\{F_1,\dots,F_m\}$ be pairwise odd-intersecting, and let
\[
    H=(\langle F_i,F_j\rangle)_{1\le i,j\le m}
\]
be its Gram matrix. Since every path has $2n$ edges,
\[
    H=J_m+I_m
    \qquad\text{over }\mathbb F_2,
\]
where $J_m$ is the all-ones matrix. The matrix $J_m+I_m$ has rank $m$ when $m$ is even and rank $m-1$ when $m$ is odd; in particular,
\[
    m-1\le\rank(H).
\]
Because $H$ is the Gram matrix of vectors in $W$,
\[
    \rank(H)\le\rank(G_W).
\]
Using \eqref{eq:rankGW} and Lemma~\ref{lem:grid-rank}, we obtain
\[
    m-1\le n(n-1),
\]
so
\[
    |\mathcal F|=m\le n(n-1)+1.
\]
\subsection{Lower bounds}
Now, we show the lower bound constructions. We do not include the proof that these families are pairwise odd intersecting as it is easy to check.

\textbf{When $n=2k$:} 
    
    Consider the family \begin{equation}
        \mathcal{F}_{2k} = \{1^{2i}0^{2j+1}10^{2k-2j-1}1^{2k-2i-1}: 0\le i,j<n\}
    \end{equation} This family has size $k^2$ and satisfies the odd intersection rule.

    \textbf{When $n=2k+1$:}
    
    We can add $2k$ more paths to the previous family.
    Consider $\mathcal{F}'_{2k}=\{0 1^{2i}0^{2j+1}10^{2k-2j-1}1^{2k-2i}:0\le i,j<k\}$ obtained from $\mathcal{F}_{2k}$. 
    Now, we add the following paths: For $0\le r<k$, let $a_r:=1^{2k+2-2r}01^{2r-2}0^{2k}1$ and $b:=01^{2k}0^{2k+1-2r}10^{2r-1}$.
    The final family \begin{equation}
        \mathcal{F}_{2k+1}=\mathcal{F}'_{2k} \cup \{a_r,b_r:0\le r<k\}
    \end{equation}
    has size $k^2+2k$ and satisfies the odd intersection rule.
\section{Further remarks}\label{sec:catalan}

Let $C_n$ denote the $n$th Catalan number. Among its many interpretations, $C_n$ counts the noncrossing perfect matchings of $[2n]$ \cite{ec2-stan}.

Let $M$ be a noncrossing perfect matching of $[2n]$. Define
\begin{equation}\label{eq:FM}
\mathcal F_M
:=
\left\{
    w\in\{E,N\}^{2n}:
    \{w_i,w_j\}=\{E,N\}
    \text{ for every }\{i,j\}\in M
\right\}.
\end{equation}
Each matching edge can be oriented independently as $EN$ or $NE$. Therefore every word in $\mathcal F_M$ contains exactly $n$ occurrences of each letter, and
\[
    |\mathcal F_M|=2^n.
\]

We claim that $\mathcal F_M$ is pairwise even-intersecting. Let $w,w'\in\mathcal F_M$, and define
\[
    \Delta_t
    :=
    \sum_{s=1}^{t}
    \bigl(\mathbf 1[w_s=N]-\mathbf 1[w'_s=N]\bigr),
    \qquad
    \Delta_0=0.
\]
At position $t$, the two words use the same lattice edge exactly when
\[
    w_t=w'_t
    \qquad\text{and}\qquad
    \Delta_{t-1}=0.
\]

Call a matching edge $\{a,b\}\in M$, with $a<b$, \emph{unchanged} if $w$ and $w'$ give it the same orientation. If it is changed, then $w$ and $w'$ use different steps at both $a$ and $b$, so neither position contributes a common edge. Suppose that $\{a,b\}$ is unchanged. Since $M$ is noncrossing, every matching edge with one endpoint strictly between $a$ and $b$ has its other endpoint strictly between $a$ and $b$ as well. Each such matching edge contributes a total of zero to the change in $\Delta$, because each word contains one $N$ and one $E$ on that edge. Consequently,
\[
    \Delta_{b-1}=\Delta_{a-1}.
\]
Thus the positions $a$ and $b$ either both contribute common edges or neither does. Common edges therefore occur in pairs indexed by the unchanged matching edges, and $|w\cap w'|$ is even.

It remains to show that different matchings give different families. For distinct $i,j\in[2n]$,
\begin{equation}\label{eq:recover-M}
\{i,j\}\in M
\quad\Longleftrightarrow\quad
w_i\ne w_j\text{ for every }w\in\mathcal F_M.
\end{equation}
The forward implication follows from the definition. If $i$ and $j$ are not matched to each other, they lie on two different matching edges, whose orientations can be chosen independently so that $w_i=w_j$. Hence \eqref{eq:recover-M} holds, and $M$ can be recovered from $\mathcal F_M$.

There are therefore at least $C_n$ distinct families of size $2^n$ satisfying the even-intersection condition. By Theorem~\ref{thm:eventown}, all of them are extremal.

\begin{conjecture}\label{conj:catalan}
There are exactly $C_n$ families of size $2^n$ in which every two distinct paths have an even number of common edges.
\end{conjecture}

The conjecture has been verified for $n\le5$. If true in general, it would give another combinatorial interpretation of the Catalan numbers.

This notion can be explored further for modular intersection results and  generalisations to paths from $(0,0)$ to $(m,n)$.

\end{document}